\documentclass[12pt,oneside,a4paper,reqno]{amsart}
\usepackage{latexsym,amsxtra,amsthm, amsmath, amsfonts, amscd, amssymb, verbatim}
\usepackage{amsfonts}
\usepackage[portrait, top=2.5 cm, bottom=3cm, left=3.5cm, right=2cm] {geometry} 

\usepackage{varioref}
\usepackage{fancyhdr}
\usepackage{longtable}
\usepackage{fancybox}
\usepackage{color, graphicx}
\usepackage{verbatim}
\usepackage{setspace}

\usepackage[latin1]{inputenc}
\usepackage[english]{babel}

\newtheorem{prop} {Proposition} [section]
\newtheorem{thm}[prop] {Theorem}
\newtheorem{defi}[prop] {Definition}
\newtheorem{lem}[prop] {Lemma}
\newtheorem{cor}[prop]{Corollary}

\newtheorem{prop-def}[prop]{Proposition-Definition}

\begin{document}

\title[rational points on complete intersections of two quadric surfaces]{Uniform bounds for rational points on complete intersections of two quadric surfaces}

\author{Manh Hung Tran}
\address{Chalmers University of Technology and University of Gothenburg \\
 Sweden}\email{manhh@chalmers.se}

\maketitle

\begin{abstract}
We give uniform upper bounds for the number of rational points of height at most $B$ on non-singular complete intersections of two quadrics in $\mathbb{P}^3$ defined over $\mathbb{Q}$. To do this, we combine determinant methods with descent arguments.
\end{abstract}

\tableofcontents

\section{\textbf{Introduction}}
Let $C$ be a non-singular complete intersection of two quadrics in $\mathbb{P}^3$ defined by $$q(x_0,x_1,x_2,x_3)=r(x_0,x_1,x_2,x_3)=0,$$ where $q$ and $r$ are quadratic forms in $\mathbb{Z}[x_0,x_1,x_2,x_3] $. Thus $C$ is of genus 1 and related to elliptic curves. We want to find uniform upper bounds for the counting function
$$N(B):=\sharp \{P\in C(\mathbb{Q}): H(P)\leq B\},$$
where the naive height function $H(P):=$max$\{|x_0|,|x_1|,|x_2|,|x_3|\}$ for $P=[x_0,x_1,x_2,x_3]$ with coprime integer values of $x_0,x_1,x_2,x_3.$ The first result of this paper is the following.

\begin{thm} Let $C$ be a non-singular complete intersection of two quadrics in $\mathbb{P}^3$ and $r$ be the rank of the Jacobian $Jac(C)$. Then for any $B\geq 3$ and any positive integer $m$ we have
$$N(B)\ll m^r \left(B^{\frac{1}{2m^2}}+m^2\right)\log B$$
uniformly in $C$, with an implied constant independent of $m$. \end{thm}

The proof follows the same strategy as in the paper [7] on non-singular cubic curves where the authors combine Heath-Brown's $p$-adic determinant method in [6] with descent theory. But we will follow the approach in [15] and replace the $p$-adic determinant method by Salberger's global determinant method [13]. Taking $m=1+[\sqrt{\log B}]$ we immediately obtain the following result.

\begin{cor}Under the condition above we have
$$N(B)\ll (\log B)^{2+r/2}$$
uniformly in $C$. \end{cor}

The upper bounds in Theorem 1.1 are uniform in the sense that the implicit constants only depend on the rank of the Jacobian. We will also use another approach to improve the uniformity and establish upper bounds which do not depend on the rank of Jac($C$). In this direction, Heath-Brown [6] obtained the bound $N(B)\ll_{\varepsilon} B^{1/2+\varepsilon}$ by using his $p$-adic determinant method. Salberger [13] proved a slightly better estimate \linebreak $N(B)\ll B^{1/2}\text{log }B.$

The aim of this paper is to improve these bounds for a class of such curves $C$ in $\mathbb{P}^3$ by using Theorem 1.1 and a refinement of the $p$-adic determinant method. We shall prove the following theorem.

\begin{thm} Let $\delta < 3/392$ and $C$ be a non-singular complete intersection in $\mathbb{P}^3$ defined by two simultaneously diagonal quadratic forms $q$ and $r$, where
$$q(x_0,x_1,x_2,x_3)=a_0 x_0^2+a_1 x_1^2 +a_2 x_2^2 + a_3 x_3^2,$$
$$r(x_0,x_1,x_2,x_3)=b_0 x_0^2+b_1 x_1^2 +b_2 x_2^2 + b_3 x_3^2$$
with integral coefficients $a_i, b_i$. Then $$N(B)\ll B^{1/2-\delta},$$
where the implicit constant depends solely on $\delta$ and not on the coefficients of $q$ and $r$. \end{thm}
This class contains examples of elliptic curves with arbitrary $j$-invariants.

\section{\textbf{Proof of Theorem 1.1}}
We shall in this section follow the approach for non-singular cubic curves in [15], where the author combined the global determinant method developed by Salberger [13] and the descent method of Heath-Brown and Testa [7]. The difference is that we now study non-singular quartic curves of genus 1 in $\mathbb{P}^3$. We first use descent to reduce the study of $N(B)$ to a counting problem for certain biprojective curves.

Let $\psi : C\times C \rightarrow \text{Jac}(C)$ be the morphism to the Jacobian of $C$ defined by \linebreak $\psi(P,Q)=[P]-[Q].$ Let $m$ be a positive integer and define an equivalence relation on $C(\mathbb{Q})$ as follows: $P\sim_m Q$ if $\psi(P,Q)\in m(\text{Jac}(C)(\mathbb{Q})).$ The number of equivalence classes is at most $16m^r$ by the theorems of Mazur and Mordell-Weil. There is therefore a class $K\subset C(\mathbb{Q})$ such that
$$N(B)\ll m^r \sharp \{P\in K: H(P)\leq B\}.$$
If we fix a point $R$ in $K$ then for any other point $P$ in $K$, there will be a further point $Q$ in $C(\mathbb{Q})$ such that $[P]=m[Q]-(m-1)[R]$ in the divisor class group of $C$. We define the curve $X=X_R$ by
$$X_R:=\{(P,Q)\in C\times C : [P]=m[Q]-(m-1)[R]\}$$
in $\mathbb{P}^3\times \mathbb{P}^3.$ Then $N(B)\ll m^r \sharp \mathcal{K}$, where
$$\mathcal{K}:= \{(P,Q)\in X(\mathbb{Q}): H(P)\leq B\}.$$
We have thus reduced the counting problem for $C$ to a counting problem for biprojective curves in $\mathbb{P}^3\times \mathbb{P}^3.$ Moreover, we can also reduce to the case where $C$ is defined by quadratic forms of small heights. We denote by $||F||$ the maximum modulus of the integral coefficients of $F$. The following result is an easy consequence of Lemma 5 in the paper of Broberg [2].

\begin{lem} Let $C$ be an integral quartic curve in $\mathbb{P}^3$ defined by two quadratic forms $q,r$ in $\mathbb{Q}[x_0,x_1,x_2,x_3]$, then either $N(B)\leq 8$ or the homogeneous ideal $I=\langle q,r \rangle $ in $\mathbb{Q}[x_0,x_1,x_2,x_3]$ can be generated by two quadratic forms $q',r'$ in $\mathbb{Z}[x_0,x_1,x_2,x_3]$ such that $\| q'\|\|r'\| \ll B^{160}.$ \end{lem}

\begin{proof} By [2], if $N(B)>8$ then $I$ can be generated by forms $q_1,...,q_t$ of degrees at most 2 such that $\prod_{i=1}^t ||q_i||\ll B^{160}.$ Since $C$ is an integral complete intersection of quadrics, it cannot be contained in a plane. So the $q_i$ are all irreducible quadratic forms. On the other hand, the intersection of any two elements $q',r'$, say, from $\{q_1,...,q_t\}$ defines a quartic curve in $\mathbb{P}^3$ which contains $C$. Hence $C$ is defined by $q'$ and $r'$. \end{proof}

Thus from now on, we may suppose that $C$ is a complete intersection defined by two quadratic forms $q,r$ in $\mathbb{Z}[x_0,x_1,x_2,x_3]$ with $\| q\|\|r\| \ll B^{160}$. We shall also need the following lemma.

\begin{lem} Let $C$ in $\mathbb{P}^3$ be a non-singular complete intersection in $\mathbb{P}^3$ defined by two quadratic forms $q,r$ in $\mathbb{Z}[x_0,x_1,x_2,x_3]$ with $\|q\|\|r\|\ll B^{160},$ and $R$ be a point in $C(\mathbb{Q}).$ Then there exists an absolute constant $A$ with the following property. Suppose that $(P,Q)$ is a point in $X_R(\mathbb{Q})$ and that $B\geq 3.$ Then if $H(P)$ and $H(R)$ are at most $B$ we have $H(Q)\leq B^A.$ \end{lem}

The proof is similar to the proof of Lemma 2.1 of [7].

\begin{proof} Let us first introduce the logarithmic height $h(P):=\log H(P)$ of a point $P$ in projective spaces $\mathbb{P}^2$ and $\mathbb{P}^3$. Note that for a point $P=[x_0,x_1,x_2]$ in $\mathbb{P}^2$ with coprime integer values of $x_0,x_1,x_2$, we define the naive height of $P$ in the same way $H(P):=\max \{|x_0|,|x_1|,|x_2|\}.$ As in [1, Section 3.3], we can choose a model for Jac$(C)$ in Weierstrass normal form $y^2=x^3+\alpha x +\beta$ such that
$$h([1,\alpha, \beta])\ll 1+\log \left( \|q\|\|r\|\right)$$
and so that
$$h_x(\psi(P,R))\ll 1+\log H(P)+\log H(R)+\log \left( \|q\|\|r\|\right)$$
and
$$\log H(P)\ll 1+h_x(\psi(P,R))+\log H(R)+\log \left( \|q\|\|r\|\right),$$
where $h_x$ is the logarithmic height of the $x$-coordinate. We also use the fact that on Jac$(C)$ the canonical height $\hat{h}$ satisfies
$$|\hat{h}(W)-h_x(W)|\ll 1+h([1,\alpha,\beta])\ll 1+\log \left( \|q\|\|r\|\right).$$
Since $\psi(P,R)=m\psi(Q,R)$ we deduce that $\hat{h}(\psi(P,R))=m^2 \hat{h}(\psi(Q,R))$. Then
$$\log H(Q)\ll 1+h_x(\psi(Q,R))+\log H(R)+\log \left( \|q\|\|r\|\right)$$
$$\ll 1+\hat{h}(\psi(Q,R))+\log H(R)+\log \left( \|q\|\|r\|\right)$$
$$=1+m^{-2}\hat{h}(\psi(P,R))+\log H(R)+\log \left( \|q\|\|r\|\right)$$
$$\ll 1+m^{-2}h_x(\psi(P,R))+\log H(R)+\log \left( \|q\|\|r\|\right)$$
$$\ll 1+ \log H(P)+\log H(R)+\log \left( \|q\|\|r\|\right)\ll \log B,$$
since $\|q\|\|r\|\leq B^{160}.$  \end{proof}

We now apply the global determinant method in [13] to $X$ and consider congruences between integral points on $X$ modulo all primes of good reduction for $C$ and $X$. It is a refinement of the $p$-adic determinant method used in [6] and [7].

We will label the points in $\mathcal{K}$ as $(P_i,Q_i)$ for $1\leq i\leq N,$ say, and fix integers $a,b\geq 1$. Let $I_1$ be the vector space of all bihomogeneous forms in $(x_0,x_1,x_2,x_3;y_0,y_1,y_2,y_3)$ of bidegree $(a,b)$ with coefficients in $\mathbb{Q}$ and $I_2$ be the subspace of such forms which vanish on $X$. Since the monomials
$$x_0^{e_0}x_1^{e_1}x_2^{e_2}x_3^{e_3}y_0^{f_0}y_1^{f_1}y_2^{f_2}y_3^{f_3}$$
with $$e_0+e_1+e_2+e_3=a \text{  and  }f_0+f_1+f_2+f_3=b$$
form a basis for $I_1$, there is a subset of monomials $\{F_1,...,F_s\}$ whose corresponding cosets form a basis for $I_1/I_2$. We will prove the following result later in Section 5.
\begin{lem} If $a,b$ and $m$ are positive integers satisfying the inequality $\frac{1}{a}+\frac{m^2}{b}<4$, then $s=4(m^2a+b).$ \end{lem}
Thus we shall always assume that $a\geq 1$ and $b\geq m^2$ to make sure that $s=4(m^2 a +b).$ Consider the $N\times s$ matrix
$$M=\left(
  \begin{array}{cccc}
    F_1(P_1,Q_1) & F_2(P_1,Q_1) & \ldots & F_s(P_1,Q_1) \\
    F_1(P_2,Q_2) & F_2(P_2,Q_2) & \ldots & F_s(P_2,Q_2) \\
    \vdots & \vdots & \ldots & \vdots \\
    F_1(P_N,Q_N) & F_2(P_N,Q_N) & \ldots & F_s(P_N,Q_N) \\
  \end{array}
\right).$$
If we can choose $a$ and $b$ such that rank$(M)<s$, then there is a non-zero column vector $\underline{c}$ such that $M\underline{c}=\underline{0}$. This will produce a bihomogeneous form $G$, say, of bidegree $(a,b)$ such that $G(P_i,Q_i)=0$ for all $1\leq i\leq N.$ The points in $\mathcal{K}$ will then lie on the variety $Y\subset \mathbb{P}^3\times \mathbb{P}^3$ given by $G=0,$ while the irreducible curve $X$ will not lie on $Y$. Thus the intersection number $X.Y$ provides an upper bound for $N$.

Now let $H$ and $H'$ be the varieties on $\mathbb{P}^3 \times \mathbb{P}^3$ given by $x_0=0$ and $y_0=0$ respectively. Then $Y$ is linearly equivalent to $aH+bH'$ and $X.Y=aX.H+bX.H'$. Further, $X.H'=4$ as a hyperplane in $\mathbb{P}^3$ intersects $C$ in 4 points and $X.H=4m^2$ since for a fixed point $P$ on $C$ there are $m^2$ pairs $(P,Q)$ on $X$. Hence
\begin{equation}
N\leq \sharp (X\cap Y)\leq X.Y=4(m^2a+b).
\end{equation}
In order to show that rank$(M)<s$, we may clearly suppose that $N\geq s$. We will show that each $s\times s$ minor det$(\Delta)$ of $M$ vanishes. Without loss of generality, let $\Delta$ be the $s\times s$ matrix formed by the first $s$ rows of $M$.
$$\Delta=\left(
  \begin{array}{cccc}
    F_1(P_1,Q_1) & F_2(P_1,Q_1) & \ldots & F_s(P_1,Q_1) \\
    F_1(P_2,Q_2) & F_2(P_2,Q_2) & \ldots & F_s(P_2,Q_2) \\
    \vdots & \vdots & \ldots & \vdots \\
    F_1(P_s,Q_s) & F_2(P_s,Q_s) & \ldots & F_s(P_s,Q_s) \\
  \end{array}
\right).$$
The main idea of the determinant method is to give an upper bound for det$(\Delta)$ and to show that it has an integral factor which is larger than this bound. It is not difficult to see that every entry in $\Delta$ has modulus at most $B^a B^{Ab}$, where $A$ is the absolute constant in Lemma 2.2. Since $\Delta$ is a $s\times s$ matrix, we get that
\begin{equation}\label{2}
  |\text{det}(\Delta)|\leq s^s B^{s(a+Ab)}.
\end{equation}
Now we find a factor of det$(\Delta)$ of the form $p^{N_p},$ where $p$ is a prime of good reduction for $C$. In order to do that, we divide $\Delta$ into blocks such that elements in each block have the same reduction modulo $p$.

Let $p$ be a prime number and $Q^*$ be a point on $C(\mathbb{F}_p),$ we then define the set
$$S(Q^*,p,\Delta )=\{(P_i,Q_i): 1\leq i \leq s, \text{ }\overline{Q_i}=Q^*\},$$
where $\overline{Q_i}$ denotes the reduction from $C(\mathbb{Q})$ to $C(\mathbb{F}_p)$. Let $E=\sharp S(Q^*,p,\Delta )$. We consider any $E\times E$ submatrix $\Delta^*=(F_j(P_i,Q_i))_{i,j}$ of $\Delta$ with all $(P_i,Q_i)$ in $S(Q^*,p,\Delta )$ and get the following result by means of Lemma 2.5 of [11].

\begin{lem} If $p$ is a prime of good reduction for $C$, then there exists a non-negative integer $\nu \geq \frac{E^2}{2}+O(E)$ such that $p^{\nu}$ divides $\det(\Delta^*).$ \end{lem}

\begin{proof} The result in [11] can also be applied to our biprojective curve as follows. The bihomogeneous monomials of bidegree $(a,b)$ will first give an embedding of the curve $X$ in $\mathbb{P}^g \times \mathbb{P}^h$ via the Veronese map, where $g=\binom{3+a}{3}-1$ and $h=\binom{3+b}{3}-1$, then in a subspace of the big projective space $P^l$ via the Segre map, where $l=(g+1)(h+1)-1=\binom{3+a}{3}\binom{3+b}{3}-1$. This proves the lemma. \end{proof}

From this lemma we obtain a factor of det$(\Delta)$ of the form $p^{N_p}$ by means of Laplace expansion. Moreover, we can use the same argument for all primes of good reduction for $C$.

\begin{lem} Let $p$ be a prime of good reduction for $C$, then there exists a non-negative integer $N_p\geq \frac{s^2}{2n_p}+ O(s)$ such that $p^{N_p} | \det(\Delta),$ where $n_p$ is the number of $\mathbb{F}_p$-points on $C(\mathbb{F}_p)$. \end{lem}

\begin{proof}  Let $P$ be a point on $C(\mathbb{F}_p)$ and $s_P$ be the number of elements in $S(P,p,\Delta ),$ then there exists from Lemma 2.4 an integer $N_P\geq \frac{s_P^2}{2}+O(s_P)$ such that $p^{N_P} | \text{det}(\Delta^*)$ for each $s_P\times s_P$ submatrix $\Delta^*=(F_j(P_i,Q_i))_{i,j}$ of $\Delta$ with all $(P_i,Q_i)$ in $S(P,p,\Delta )$.

If we apply this to all points on $C(\mathbb{F}_p)$ and use Laplace expansion, then we get that $p^{N_p} | \text{det}(\Delta)$ for
$$N_p=\sum_P N_P=\frac{1}{2}\sum_P {s_P}^2+O(s)\geq \frac{s^2}{2n_p}+O(s)$$
in case $C$ has good reduction at $p$.\end{proof}

We now give a bound for the product of primes of bad reduction for $C$. Since we can assume that $\|q\|\|r\|\ll B^{160}$, the discriminant $D_C$ of $C$ will satisfy log$|D_C|\ll$ log $B$. It follows that log $\Pi_C\ll$ log $B$, where $\Pi_C$ is the product of all primes of bad reduction for $C$. We have therefore the following bound.

\begin{lem} Suppose that $\|q\|\|r\|\ll B^{160}.$ The product $\Pi_C$ of all primes of bad reduction for $C$ satisfies $\log \Pi_C=O(\log B).$ \end{lem}
We need one more lemma from [13] (see Lemma 1.10).
\begin{lem} Let $\Pi>1$ be an integer and $p$ run over all prime factors of $\Pi$. Then $$ \sum_{p|\Pi}\frac{{\log}\text{ }p}{p}\leq {\log \log }\text{ }\Pi+2.$$ \end{lem}
We now use the previous lemmas to prove that det$(\Delta)$ vanishes if $s$ is large enough. Let $\Pi_C$ be the product of all primes $p$ of bad reduction for $C$, then

\begin{equation}\label{3}
\displaystyle \sum_{p|\Pi_C}\frac{\text{log }p}{p}\leq \text{log }\text{log }B+O(1)
\end{equation}
by Lemma 2.6 and Lemma 2.7. We apply Lemma 2.5 to the primes $p\leq s$ of good reduction for $C$ and write $\displaystyle {\sum_{p\leq s}}^*$ for a sum over these primes. We then obtain a positive factor $T$ of det$(\Delta)$ which is relatively prime to $\Pi_C$ such that
$$\displaystyle\text{log }T\geq \frac{s^2}{2}{\sum_{p\leq s}}^* \frac{\text{log }p}{n_p}+O(s){\sum_{p\leq s}}^* \text{log }p.$$
The last term is $O(s^2)$ since $\sum_{p\leq s}\text{log }p=O(s)$ (see [14], p. 31). Also,
$$\displaystyle \frac{\text{log }p}{n_p}\geq \frac{\text{log }p}{p}-\frac{(n_p-p)\text{log }p}{p^2}.$$
Moreover, it is a well-known result of Hasse that $n_p=p+O(\sqrt{p})$ for a prime $p$ of good reduction for $C$. Thus we conclude that
$$\displaystyle \frac{\text{log }p}{n_p}\geq \frac{\text{log }p}{p}+O\left(\frac{\text{log }p}{p^{3/2}}\right)$$
for all primes $p$ of good reduction for $C.$ Therefore,
$${\sum_{p\leq s}}^* \frac{\text{log }p}{n_p}\geq {\sum_{p\leq s}}^* \frac{\text{log }p}{p}+O(1)$$
and hence $$\text{log }T\geq \frac{s^2}{2}{\sum_{p\leq s}}^* \frac{\text{log }p}{p}+O(s^2).$$
But by (3),
$$\displaystyle\sum_{p\leq s} \frac{\text{log }p}{p}-{\sum_{p\leq s}}^* \frac{\text{log }p}{p}\leq \text{log }\text{log }B+O(1)$$
and $\displaystyle \sum_{p\leq s} \frac{\text{log }p}{p}=\text{log }s+O(1)$ (see [14], p. 14). Hence,
\begin{equation}\label{4}
\displaystyle \text{log }T\geq \frac{s^2}{2}\text{log}\left(\frac{s}{\text{log }B}\right)+O(s^2).
\end{equation}
Thus from (2) and (4) we obtain
$$ \displaystyle \text{log}\left(\frac{|\text{det}(\Delta)|}{T}\right)\leq s\text{log }s+s\text{log }B^{a+Ab}-\frac{s^2}{2}\text{log}\left(\frac{s}{\text{log }B}\right)+O(s^2)$$
$$=\displaystyle \frac{s^2}{2}\left(\log B^{\frac{2(a+Ab)}{s}}-\text{log}\left(\frac{s}{\text{log }B}\right)\right)+O(s^2).$$
There is therefore an absolute constant $u\geq 1$ such that
$$\displaystyle \text{log}\left(\frac{|\text{det}(\Delta)|}{T}\right)\leq \frac{s^2}{2}\left(\log B^{\frac{2(a+Ab)}{s}}-\text{log}\left(\frac{s}{u\text{log }B}\right)\right).$$
If
\begin{equation}
\displaystyle s>u B^{\frac{2(a+Ab)}{s}}\text{log }B
\end{equation}
we have in particular that $\log \left(\frac{|\det (\Delta)|}{T}\right)<0$ and hence det$(\Delta)=0$ as $\frac{|\det (\Delta)|}{T}\in \mathbb{Z}_{\geq 0}.$

Recall that by Lemma 2.3 we have that $s=4(m^2 a +b)$ if $a\geq 1$ and $b\geq m^2.$ We now choose $b=m^2$ and
$$a=1+\left[\frac{uB^{\frac{1}{2m^2}}\log B}{m^2}+A\log B\right].$$
Then
$$u B^{\frac{2(a+Ab)}{s}}\text{log }B=u B^{\frac{a+Am^2}{2m^2(a+1)}}\text{log }B $$
$$< u B^{\frac{1}{2m^2}}B^{\frac{A}{2a}}\log B<s.$$
Thus (5) holds such that det$(\Delta)=0$ for any $s\times s$ minor det$(\Delta)$ of $M$. As rank$(M)<s$, there is thus a bihomogeneous form in $\mathbb{Q}[x_0,x_1,x_2,x_3;y_0,y_1,y_2,y_3]$ which vanishes at all $(P_i,Q_i)\in X(\mathbb{Q})$, $1\leq i\leq N,$ with $H(P_i)\leq B$ but not everywhere on $X$. Hence (see (1))
$$N\leq 4(m^2 a+b)\ll \left(B^{\frac{1}{2m^2}}+m^2\right)\log B$$
$$\Rightarrow N(B)\ll m^r \left(B^{\frac{1}{2m^2}}+m^2\right)\log B.$$
This completes the proof of Theorem 1.1.

\section{\textbf{Savings for curves of large height}}

The main goal of this section is to prove Theorem 3.1, which is the key result to obtain Theorem 1.3. For a curve $C$ in $\mathbb{P}^3$ given by a non-singular complete intersection of two quadrics, Heath-Brown [6] showed that $N(B)\ll_{\varepsilon} B^{1/2+\varepsilon}$ for the number $N(B)$ of rational points of height at most $B$ on $C$ by his $p$-adic determinant method. We will use a refinement of that method where we make use of extra factors in the determinant which come from the coefficients of the quadratic forms defining $C$. To do this, we first need to define a height function on a parameter variety of such quartic curves. Unfortunately we do not have any improvement for general non-singular complete intersections of two quadrics in $\mathbb{P}^3$. In this section we will therefore only discuss the case where $C$ is a non-singular complete intersection defined by two simultaneously diagonal quadratic forms.

Let $V$ be the 4-dimensional vector space of diagonal quadratic forms \linebreak $a_0 x_0^2+a_1 x_1^2+a_2 x_2^2+a_3 x_3^2$ with coefficients in $\mathbb{Q}$. Then if $q,r\in V$ are linearly independent, we get a complete intersection $q=r=0$ in $\mathbb{P}^3$ which only depends on the vector space $W\subset V$ spanned by $q$ and $r$. As the 2-dimensional subspaces of $W$ are parametrized by the Grassmannian $\text{Gr}(2,V)$, we therefore get a universal family $\mathcal{F}\subset \mathbb{P}^3 \times \text{Gr}(2,V)$ of quartic space curves $C\subset \mathbb{P}^3$ over $\text{Gr}(2,V)$. If we use Pl\"ucker \linebreak coordinates for $\text{Gr}(2,V)$, then $W=\langle a_0 x_0^2+a_1 x_1^2+a_2 x_2^2+a_3 x_3^2, b_0 x_0^2+b_1 x_1^2+b_2 x_2^2+b_3 x_3^2\rangle$ is uniquely determined by the sixtuple $$d_{ij}=\left|
                                                                    \begin{array}{cc}
                                                                      a_i & a_j \\
                                                                      b_i & b_j \\
                                                                    \end{array}
                                                                  \right|,
\text{ }0\leq i<j\leq 3$$ in $\mathbb{P}^5$. We will therefore define the height $H(C)$ of the quartic curve
$$a_0 x_0^2+a_1 x_1^2+a_2 x_2^2+a_3 x_3^2=b_0 x_0^2+b_1 x_1^2+b_2 x_2^2+b_3 x_3^2=0,$$
with integral coefficients $a_i,b_i,$ to be the height of the sixtuple $(d_{ij};0\leq i<j\leq 3)$ in $\mathbb{P}^5(\mathbb{Q})$. We have thus
$$\displaystyle H(C):=\max_{0\leq i<j\leq 3}(|d_{ij}|)/\gcd_{0\leq i<j \leq 3}(d_{ij}).$$
The main result of this section is the following

\begin{thm} Let $C$ be as in Theorem 1.3, we have
$$N(B)\ll_{\varepsilon} B^{1/2+\varepsilon}/H(C)^{1/8}+\log B +1.$$ \end{thm}
This is an analog of Proposition 2.1 in Ellenberg and Venkatesh [4] where the authors showed a similar estimate for irreducible hypersurfaces in $\mathbb{P}^n$. Before proving Theorem 3.1, we will need various preliminary results for non-singular quartic space curves defined by two simultaneously diagonal quadratic forms.

\begin{defi} We will call a pair of quadratic forms $q=a_0 x_0^2+a_1 x_1^2+a_2 x_2^2+a_3 x_3^2$, $r=b_0 x_0^2+b_1 x_1^2+b_2 x_2^2+b_3 x_3^2$ in $\mathbb{Z}[x_0,x_1,x_2,x_3]$ primitive if $\gcd_{i\neq j}(a_ib_j-a_jb_i)=1$. \end{defi}

We can assume that $C\subset \mathbb{P}^3$ is defined by a primitive pair $(q,r)$ in $\mathbb{Z}[x_0,x_1,x_2,x_3]$ by the following lemma.

\begin{lem} Let $\underline{a}=(a_0,a_1,a_2,a_3)$, $\underline{b}=(b_0,b_1,b_2,b_3)\in \mathbb{Z}^4$ be quadruples with $a_i b_j -a_j b_i \neq 0$ for some $i,j$. Then there exists $\underline{c}=(c_0,c_1,c_2,c_3)$, $\underline{d}=(d_0,d_1,d_2,d_3)\in \mathbb{Z}^4$ such that $\gcd_{i\neq j} (c_id_j-c_jd_i)=1$ and such that $\underline{c}$ and $\underline{d}$ span the same 2-dimensional vector space as $\underline{a}$ and $\underline{b}$. \end{lem}

\begin{proof} Let $W\subset \mathbb{Q}^4$ be the vector space spanned by $\underline{a}$ and $\underline{b}$ and $L=W\cap \mathbb{Z}^4$. Then $L$ is a free $\mathbb{Z}$-module of rank 2 and any two generators $\underline{c}$ and $\underline{d}$ will satisfy the conditions. \end{proof}

Hence we only need to prove Theorem 3.1 for curves defined by primitive pairs of quadratic forms. The benefit of being primitive is the following result.

\begin{lem} Let $q=a_0 x_0^2+a_1 x_1^2+a_2 x_2^2+a_3 x_3^2$, $r=b_0 x_0^2+b_1 x_1^2+b_2 x_2^2+b_3 x_3^2$  in  $\mathbb{Z}[x_0,x_1,x_2,x_3]$ be a primitive pair. Let $\underline{x}=(x_0,x_1,x_2,x_3),$ $\underline{y}=(y_0,y_1,y_2,y_3)\in \mathbb{Z}^4$ be such that  $q(\underline{x})=r(\underline{x})=0$ and $q(\underline{y})=r(\underline{y})=0.$ Then there exists an integer $\lambda$ such that
$$|x_k^2 y_l^2 - x_l^2 y_k^2|=\lambda |a_ib_j-a_jb_i|$$
for any  $i,j,k,l\in \{0,1,2,3\}$ with $\{i,j,k,l\}=\{0,1,2,3\}.$ \end{lem}

\begin{proof} Let $W' \subset \mathbb{Q}^4$ be the 2-dimensional subspace defined by the two equations
$$\left\{
    \begin{array}{ll}
      a_0 z_0+a_1 z_1+a_2 z_2+a_3 z_3 =0\\
      b_0 z_0+b_1 z_1+b_2 z_2+b_3 z_3 =0
    \end{array}
  \right.
.$$
Then $(x_0^2,x_1^2,x_2^2,x_3^2),$ $(y_0^2,y_1^2,y_2^2,y_3^2)\in W'.$ If these quadruples are linearly independent, then by the relation between Grassmann coordinates and dual Grassmann coordinates in [8, p. 294-297] we get that the sixtuples
$$(x_0^2 y_1^2 - x_1^2 y_0^2,x_0^2 y_2^2 - x_2^2 y_0^2,x_0^2 y_3^2 - x_3^2 y_0^2,x_1^2 y_2^2 - x_2^2 y_1^2,x_1^2 y_3^2 - x_3^2 y_1^2,x_2^2 y_3^2 - x_3^2 y_2^2)$$
and $(a_2 b_3 - a_3 b_2,a_1 b_3 - a_3 b_1,a_1 b_2 - a_2 b_1,a_0 b_3 - a_3 b_0,a_0 b_2 - a_2 b_0,a_0 b_1 - a_1 b_0)$ will define the same rational point on $\mathbb{P}^5$ (up to signs of the coordinates). Hence the statement follows from the primitivity of $(\underline{a},\underline{b}).$ \end{proof}

We are now ready to prove Theorem 3.1 by using the $p$-adic determinant method.

\text{}\\
\emph{Proof of Theorem 3.1.} The idea is to divide all rational points of height at most $B$ on $C$ into congruence classes modulo some prime number $p$ of good reduction for $C$ and then count points in each class. By Hasse's theorem, there are then at most $p+1+2\sqrt{p}$ congruence classes (mod $p$).

Since $C$ is a non-singular curve of genus 1 and degree 4 in $\mathbb{P}^3$, we have by the \linebreak Riemann-Roch theorem that $\dim H^0 (C, \mathcal{O}_C (k))=4k$ for all positive integer $k$. Hence, as the morphism $H^0 (\mathbb{P}^3, \mathcal{O}_{\mathbb{P}^3} (k))\rightarrow H^0 (C, \mathcal{O}_C (k))$ is surjective (see Hartshorne [5, p. 188]), its homogeneous coordinate ring
$$\mathbb{Q}[x,y,z,t]/(q,r)=\bigoplus_{k\geq 0}S_k$$
satisfies dim$_\mathbb{Q}S_k=4k$ for all $k\geq 1$.

Let $p$ be a prime of good reduction for $C$, we then denote by $N(B,p,P)$ the number of points of height at most $B$ in $C(\mathbb{Q})$ which specialise to $P$ on $C(\mathbb{F}_p)$. For a given degree $2k$, we first fix $8k$ monomials $\{f_j\}$, $1\leq j\leq 8k$ of degree $2k$ which form a basis for $S_{2k}$. Our goal is to prove that det$(M_{2k})=0$ for any $8k\times 8k$-matrix $M_{2k}=(f_j(P_i))_{i,j}$, where $\{P_i\}$, $1\leq i \leq 8k$ are $8k$ points counted by $N(B,p,P)$. Note that we consider monomials of degree $2k$ instead of $k$ and we will see why. If we can choose $p$ such that det$(M_{2k})=0$ for all such sets $\{P_i\}$, then there exists a homogeneous polynomial $G=c_1f_1+...+c_{8k}f_{8k}$ of degree $2k$ which contains all the points counted by $N(B,p,P)$ but which does not contain $C$. By the theorem of B\'ezout, we have then that $N(B,p,P)\leq 8k$ for any point $P$ on $C(\mathbb{F}_p)$.

To get the vanishing of det$(M_{2k})$, we first give an upper bound and then a factor of the integer det$(M_{2k})$ which is larger than the bound. Since all the points are of height at most $B$, we get the following upper bound by using Hadamard's inequality:
\begin{equation}
|\det(M_{2k})|\leq (8k)^{4k}B^{16k^2}.
\end{equation}
To find a factor of det$(M_{2k})$, we may after elementary row operations in $M_{2k}$ over $\mathbb{Z}_p$ arrange such that all elements in the $i$-th row is divisible by $p^{i-1}$ (see the proofs of [11, Lemma 2.4] and [6, Theorem 14]). Hence
\begin{equation}
p^{4k(8k-1)}| \text{det}(M_{2k}).
\end{equation}
There are also other factors of det$(M_{2k})$ coming from the height of $C$.

\begin{prop} Let $C$ in $\mathbb{P}^3$ be a non-singular complete intersection defined by two quadratic forms $q=a_0x_0^2+a_1x_1^2+a_2x_2^2+a_3x_3^2$ and $r=b_0x_0^2+b_1x_1^2+b_2x_2^2+b_3x_3^2$ with integral coefficients. Then for any positive integer $k$, there exists a basis $\{f_1,...,f_{8k}\}$ of $S_{2k}$ such that the determinant of $M_{2k}^*=(f_j(P_i))_{i,j}$ is divisible by $H(C)^{4k^2-4k+1}$ for arbitrary $8k$ rational points $\{P_i\}$ on $C$. \end{prop}

The proof of Proposition 3.5 is the most technical part of this paper. We first recall a well-known result from linear algebra.

\begin{lem} [\emph{Vandermonde determinant}]
$$\left|
    \begin{array}{ccccc}
      \alpha_1^k & \alpha_1^{k-1}\beta_1 & ... & \alpha_1\beta_1^{k-1} & \beta_1^k \\
      \alpha_2^k & \alpha_2^{k-1}\beta_2 & ... & \alpha_2\beta_2^{k-1} & \beta_2^k \\
      \vdots & \vdots & ... & \vdots & \vdots \\
      \alpha_{k+1}^k & \alpha_{k+1}^{k-1}\beta_{k+1} & ... & \alpha_{k+1}\beta_{k+1}^{k-1} & \beta_{k+1}^k \\
    \end{array}
  \right|
  =\prod_{1\leq i <j \leq k+1}(\alpha_i\beta_j-\alpha_j\beta_i).
$$
\end{lem}

\text{}\\
\begin{proof}[Proof of Proposition 3.5.] By Lemma 3.3 we may assume that $(q,r)$ is a primitive pair such that the height $H(C)$ of $C$ is equal to $\max_{0\leq i<j \leq 3}(|a_ib_j-a_jb_i|).$ Suppose, without loss of generality, that $H(C)=|a_2b_3-a_3b_2|$. Then we choose the following basis for $S_{2k}$
$$\underbrace{x_0^{2k},x_0^{2k-2}x_1^2,...,x_1^{2k}}_{k+1},\underbrace{x_0^{2k-1}x_1,x_0^{2k-3}x_1^3,...,x_0x_1^{2k-1}}_{k},$$
$$\underbrace{x_0^{2k-1}x_2,x_0^{2k-3}x_1^2x_2,...,x_0x_1^{2k-2}x_2}_{k},\underbrace{x_0^{2k-2}x_1x_2,x_0^{2k-4}x_1^3x_2,...,x_1^{2k-1}x_2}_{k},$$
$$\underbrace{x_0^{2k-1}x_3,x_0^{2k-3}x_1^2x_3,...,x_0x_1^{2k-2}x_3}_{k},\underbrace{x_0^{2k-2}x_1x_3,x_0^{2k-4}x_1^3x_3,...,x_1^{2k-1}x_3}_{k},$$
$$\underbrace{x_0^{2k-2}x_2x_3,x_0^{2k-4}x_1^2x_2,...,x_1^{2k-2}x_2x_3}_{k},\underbrace{x_0^{2k-3}x_1x_2x_3,x_0^{2k-5}x_1^3x_2x_3,...,x_0x_1^{2k-3}x_2x_3}_{k-1}.$$
We denote by $x_{0i}^{k_0} x_{1i}^{k_1} x_{2i}^{k_2} x_{3i}^{k_3}$ the value of the monomial $x_0^{k_0} x_1^{k_1} x_2^{k_2} x_3^{k_3}$ at $P_i$. Using Laplace expansion along the first $k+1$ columns of det$(M_{2k}^*)$, we obtain that det$(M_{2k}^*)$ is a sum of $\binom{8k}{k+1}$ terms. For each of these terms, we use Laplace expansion along the first $k$ columns of the bigger matrix. We continue this process together with the order of the basis $\{f_1,f_2,...,f_{8k}\}$ above and make use of Lemma 3.6. We then conclude that det$(M_{2k}^*)$ can be written as a sum of $(8k)!$ terms such that each of these terms is divisible by (up to an order of $x_{0i}, x_{1i}$ when $i$ runs from $1$ to $8k$)
\begin{equation}
\prod_{\Omega}(x_{0i}^2x_{1j}^2-x_{0j}^2x_{1i}^2),
\end{equation}
where
$$\displaystyle \Omega=\{1\leq i<j\leq k+1\}\cup \{k+2\leq i<j\leq 2k+1\}\cup \{2k+2\leq i<j\leq 3k+1\}$$
$$\cup \{3k+2\leq i<j\leq 4k+1\}\cup \{4k+2\leq i<j\leq 5k+1\}\cup \{5k+2\leq i<j\leq 6k+1\}$$
$$\cup \{6k+2\leq i<j\leq 7k+1\}\cup \{7k+2\leq i<j\leq 8k\}.$$
The appearance of terms of the form $x_{0i}^2x_{1j}^2-x_{0j}^2x_{1i}^2$ is the reason why we considered monomials of degree $2k$ instead of $k.$ Thus we see from Lemma 3.4 that det$(M_{2k}^*)$ is divisible by $(a_2b_3-a_3b_2)^n$, where
$$n=\binom{k+1}{2}+6\binom{k}{2}+\binom{k-1}{2}=4k^2-4k+1.$$
This proves the proposition.\end{proof}

We now use this proposition to choose a basis $\{f_1,...,f_{8k}\}$ of $S_{2k}$ such that det$(M_{2k})$ is divisible by $H(C)^{4k^2-4k+1}$. This factor is relatively prime to the factor $p^{4k(8k-1)}$ in (7) as $H(C)$ is not divisible by any prime of good reduction for $C$. Hence we get that
\begin{equation}
p^{4k(8k-1)}H(C)^{4k^2-4k+1}| \text{det}(M_{2k}).
\end{equation}
From (6) and (9) we see that if $p$ satisfies the inequality
\begin{equation}
p> 8k^{\frac{1}{8k-1}}B^{\frac{4k}{8k-1}}/H(C)^{\frac{4k^2-4k+1}{4k(8k-1)}}
\end{equation}
then det$(M_{2k})=0$. Thus $N(B,p,P)\leq 8k$ for any point $P$ on $C(\mathbb{F}_p)$ and for any prime $p$ of good reduction for $C$ satisfying (10). The following lemma shows the existence of such a prime.

\begin{lem} For any integer $k\geq 1$, there is a prime $p$ of good reduction for $C$ such that
$$2B^{\frac{4k}{8k-1}}/H(C)^{\frac{4k^2-4k+1}{4k(8k-1)}}<p\ll B^{\frac{4k}{8k-1}}/H(C)^{\frac{4k^2-4k+1}{4k(8k-1)}}+1+\log B.$$ \end{lem}
\begin{proof} Since we are assuming that $\|q\|\|r\|\ll B^{160},$ the discriminant $D_C$ of $C$ will satisfy $\log |D_C|\ll 1+\log B$. The number of primes of bad reduction for $C$ is then at most
\begin{equation}
\omega (6|D_C|)\ll \frac{\log |D_C|}{\log \log |D_C|}\ll \frac{1+\log B}{\log (1+\log B)},
\end{equation}
where $\omega (n)$ denotes the number of prime divisors of $n$. However if $A$ is sufficient large there are at least $A/(2\log A)$ primes between $A$ and $2A.$ There is thus from (11) an absolute constant, $c_0$ say, such that any range $(A,2A]$ with $A\geq c_0 (1+\log B)$ contains a prime $p$ of good reduction. To complete the proof of the lemma we just need to take
$$A=2B^{\frac{4k}{8k-1}}/H(C)^{\frac{4k^2-4k+1}{4k(8k-1)}} +c_0 (1+\log B).$$ \end{proof}
We may now complete the proof of Theorem 3.1. Let $p$ be a prime satisfying \linebreak Lemma 3.7 and note that $8k^{\frac{1}{8k-1}}<2$ for all $k\geq 1$. We then get
$$N(B)\leq \sum_{P\in C(\mathbb{F}_p)} N(B,p,P)\leq \sum_{P\in C(\mathbb{F}_p)} 8k\ll_k  p$$
$$\ll_k B^{\frac{4k}{8k-1}}/H(C)^{\frac{4k^2-4k+1}{4k(8k-1)}}+1+\log B.$$
If we now let $k$ go to infinity then we obtain Theorem 3.1.

\section{\textbf{A uniform bound for quartic space curves}}

The aim of this section is to complete the proof of Theorem 1.3. To do this, we prove a lower bound for the height $H(C)$ in terms of the discriminant of Jac$(C)$ and then use the same basic dichotomy as in the two articles [4] and [7]. For curves of small height we use descent and the determinant method. To sum over the descent classes we need upper estimates for the rank of Jac$(C)$ in terms of its discriminant. For curves of large height we use a refinement of the determinant method where we make use of extra factors in the determinant which come from the coefficients of the quadratic forms defining $C$.

Let $C$ be a curve as in Theorem 1.3, the discriminant $D$ of Jac$(C)$ can be computed by means of the formulas in [1, Sections 3.1 and 3.3]. This gives
$$D=\displaystyle 2^{-8}\prod_{0\leq i\neq j\leq 3}(a_ib_j-a_jb_i).$$
If $C$ is defined by a primitive pair of quadratic forms, we have therefore
\begin{equation}
|D|\leq H(C)^{12}.
\end{equation}
We now use a standard 2-descent argument as in Brumer - Kramer [3] to bound the rank $r$ of Jac($C$) in terms of $|D|$. One can prove that for any $c>1/(2\log 2)$ we have
$$r<c \text{ log}|D|+\text{O}_{\varepsilon}(1).$$
This is discussed by Ellenberg and Venkatesh [4, p. 2177]. In Theorem 1.1, if we take $m=2$ then
\begin{equation}
N(B)\ll 2^r B^{1/8}\log B \ll_{\varepsilon} |D|^{1/2+\varepsilon}B^{1/8}\log B.
\end{equation}
From (12) and (13) we obtain that
\begin{equation}
N(B)\ll_{\varepsilon}H(C)^{6+\varepsilon}B^{1/8}\log B.
\end{equation}
Comparing (14) with Theorem 3.1 we see that the worst case is that in which \linebreak $H(C)=B^{3/49}$. We then obtain Theorem 1.3.

\section{\textbf{Proof of Lemma 2.3}}
We shall in this section prove the remaining Lemma 2.3. For any positive integers $a,b$, we denote by $(a,b)$ the divisor $aH+bH'$, where $H$ and $H'$ are the varieties in $\mathbb{P}^3 \times \mathbb{P}^3$ given by $x_0=0$ and $y_0=0$ respectively. We also recall that for any point $R\in C$,
$$X_R=\{(P,Q)\in C\times C : [P]=m[Q]-(m-1)[R]\}.$$
We first need the following result.

\begin{lem} Let $R$ be a point of $C$ and suppose that $a,b$ and $m$ are positive integers satisfying $\frac{1}{a}+\frac{m^2}{b}<4.$ Then the restriction of global sections
$$H^0 (\mathbb{P}^3 \times \mathbb{P}^3, \mathcal{O}_{\mathbb{P}^3 \times \mathbb{P}^3}(a,b))\rightarrow H^0 (X_R, \mathcal{O}_{X_R}(a,b))$$
is surjective and the dimension of $H^0 (X_R, \mathcal{O}_{X_R}(a,b))$ is $4(m^2 a+b).$ \end{lem}

It follows from the lemma that the quotient space $I_1/I_2$ defined before Lemma 2.3 may be identified with $H^0 (X_R, \mathcal{O}_{X_R}(a,b))$. It is thus a vector space of dimension $4(m^2a+b)$ spanned by bihomogeneous monomials of bidegree $(a,b).$ This completes the proof of Lemma 2.3.

\begin{proof} [Proof of Lemma 5.1.] We use arguments similar to those in the proof of Lemma 5.1 of Heath-Brown and Testa [7] where they proved a similar result for non-singular plane cubic curves in three steps.

Let $Y$ be a subvariety of $\mathbb{P}^3 \times \mathbb{P}^3$, $D\subset Y$ be an effective divisor and $\mathcal{L}$ be the restriction to $Y$ of the line bundle $\mathcal{O}_{\mathbb{P}^3 \times \mathbb{P}^3}(a,b)$. There is then a short exact sequence
\begin{equation}
0\rightarrow \mathcal{L}(-D)\rightarrow \mathcal{L}\rightarrow \mathcal{L}|_D\rightarrow 0
\end{equation}
of sheaves on $Y$. From the long exact cohomology sequence associated to (15), it follows that if the cohomology group $H^1 (Y,\mathcal{L}(-D))$ vanishes, then the restriction of global sections
$$H^0 (Y,\mathcal{L})\rightarrow H^0 (D, \mathcal{L}|_D)$$
is surjective.

\text{}\\
\textbf{Step 1:} from $\mathbb{P}^3 \times \mathbb{P}^3$ to $C\times \mathbb{P}^3.$ Let $\mathcal{I}_{C\times \mathbb{P}^3}$ be the ideal of functions on $\mathbb{P}^3 \times \mathbb{P}^3$ that vanish on $C\times \mathbb{P}^3$. The sequence (15) becomes
$$0\rightarrow \mathcal{I}_{C\times \mathbb{P}^3} \otimes \mathcal{O}_{\mathbb{P}^3 \times \mathbb{P}^3}(a,b)\rightarrow \mathcal{O}_{\mathbb{P}^3\times \mathbb{P}^3}(a,b)\rightarrow \mathcal{O}_{C\times \mathbb{P}^3}(a,b)\rightarrow 0.$$
The vanishing of $H^1 (\mathbb{P}^3 \times \mathbb{P}^3 , \mathcal{I}_{C\times \mathbb{P}^3}\otimes \mathcal{O}_{\mathbb{P}^3 \times \mathbb{P}^3}(a,b))$ can be obtained, if $a>0$ and $b>-4$, from a resolution of the ideal sheaf $\mathcal{I}_{C\times \mathbb{P}^3}$
\begin{equation}
0\rightarrow \mathcal{O}_{\mathbb{P}^3 \times \mathbb{P}^3}(-4,0)\rightarrow \mathcal{O}_{\mathbb{P}^3 \times \mathbb{P}^3}(-2,0)\oplus \mathcal{O}_{\mathbb{P}^3 \times \mathbb{P}^3}(-2,0) \rightarrow \mathcal{I}_{C\times \mathbb{P}^3}\rightarrow 0.
\end{equation}
Indeed, taking the long cohomology sequence associated to (16) (after tensoring with $\mathcal{O}_{\mathbb{P}^3 \times \mathbb{P}^3}(a,b)$) we get
$$...\rightarrow H^1 (\mathbb{P}^3 \times \mathbb{P}^3 , \mathcal{O}_{\mathbb{P}^3 \times \mathbb{P}^3}(a-4,b))\rightarrow H^1 (\mathbb{P}^3 \times \mathbb{P}^3 , \mathcal{O}_{\mathbb{P}^3 \times \mathbb{P}^3}(a-2,b)\oplus \mathcal{O}_{\mathbb{P}^3 \times \mathbb{P}^3}(a-2,b))\rightarrow $$
\begin{equation}
\rightarrow H^1 (\mathbb{P}^3 \times \mathbb{P}^3 , \mathcal{I}_{C\times \mathbb{P}^3}\otimes\mathcal{O}_{\mathbb{P}^3 \times \mathbb{P}^3}(a,b))\rightarrow H^2 (\mathbb{P}^3 \times \mathbb{P}^3 , \mathcal{O}_{\mathbb{P}^3 \times \mathbb{P}^3}(a-4,b))\rightarrow ...
\end{equation}
The vanishing of the 1st, 2nd and the last terms of (17) follows from the Kodaira Vanishing Theorem [10] if $a>0$ and $b>-4$. We thus obtain, if $a>0$ and $b>-4,$ the vanishing of $H^1 (\mathbb{P}^3 \times \mathbb{P}^3 , \mathcal{I}_{C\times \mathbb{P}^3}\otimes\mathcal{O}_{\mathbb{P}^3 \times \mathbb{P}^3}(a,b)).$

\text{}\\
\textbf{Step 2:} from $C\times \mathbb{P}^3$ to $C\times C.$ As above, $H^1 (C\times \mathbb{P}^3, \mathcal{I}_{C\times C}\otimes\mathcal{O}_{C\times \mathbb{P}^3}(a,b))$ vanishes if $a>0$ and $b>0$.

\text{}\\
\textbf{Step 3:} from $C\times C$ to $X_R.$ The curve $X_R$ is a divisor on $C\times C$ and (15) becomes
$$0\rightarrow \mathcal{O}_{C\times C}((a,b)-X_R)\rightarrow \mathcal{O}_{C\times C}(a,b)\rightarrow \mathcal{O}_{X_R}(a,b) \rightarrow 0$$
in this case. Note that $C\times C$ is isomorphic to an abelian surface over $\overline{\mathbb{Q}}$ and therefore every effective divisor on $C\times C$ is nef. The vanishing of the group
$$H^1 (C\times C , \mathcal{O}_{C\times C}((a,b)-X_R))$$
is thus a consequence of the Kawamata-Viehweg Vanishing Theorem [9,16] if the \linebreak inequalities $(0,1)((a,b)-X_R)>0$ and $((a,b)-X_R)^2>0$ hold. We have
$$(0,1)((a,b)-X_R)=a(0,1)(1,0)+b(0,1)^2-(0,1)X_R=4(4a-1).$$
Here we use the facts that $(0,1)(1,0)=16$ since a general hyperplane in $\mathbb{P}^3$ intersects $C$ in four points, that $(0,1)^2=0$ as a general line in $\mathbb{P}^3$ is disjoint from $C$ and that $(0,1)X_R=4$ since for a fixed point $Q$ on $C$ there is a unique pair $(P,Q)$ on $X_R$.

To compute $((a,b)-X_R)^2,$ we observe that $(1,0)X_R=4m^2$ since for a fixed point $P$ on $C$ there are $m^2$ pairs $(P,Q)$ on $X_R$. Moreover, $(X_R)^2=0$ since for all $R,R'\in C$ the curves $X_R$ and $X_{R'}$ are algebraic equivalent and if $(m-1)([R]-[R'])\neq 0,$ then the curves $X_R$ and $X_{R'}$ are disjoint. Hence
$$((a,b)-X_R)^2=(a,b)^2-2(a,b)X_R+(X_R)^2=8ab\left(4-\frac{1}{a}-\frac{m^2}{b}\right)$$
and the first part of the lemma is obtained.

We now compute the dimension of $H^0 (X_R, \mathcal{O}_{X_R}(a,b))$. Here $X_R$ is smooth of genus one since the projection of the curve $X_R \subset C\times C$ onto the second factor is an isomorphism. As the line bundle $\mathcal{O}_{X_R}(a,b)$ on $X_R$ has degree
\begin{equation}
X_R.(a,b)=4(m^2 a +b)>0,
\end{equation}
we get that $H^1 (X_R, \mathcal{O}_{X_R}(a,b))=0$ and then from the Riemann-Roch formula that the dimension of $H^0 (X_R, \mathcal{O}_{X_R}(a,b))$ is $4(m^2 a+b).$ This completes the proof of Lemma 5.1.
\end{proof}

\section*{\textbf{Acknowledgement}}

I would like to thank my supervisor Per Salberger for introducing me to the problem and giving me many important ideas and comments. I am also grateful to Dennis Eriksson for useful discussions.

\end{document}